\documentclass[11pt,a4paper, envcountsame]{amsart}

\usepackage[usenames,dvipsnames]{color}

 \usepackage[colorlinks,citecolor=blue,urlcolor=blue, linkcolor=blue]{hyperref}

 \usepackage{amsmath, amssymb, xspace}
 \usepackage{graphicx}

 \usepackage[all]{xy}

\usepackage{tikz}
\usetikzlibrary{arrows,snakes,positioning,backgrounds,shadows}

\newtheorem{theorem}{Theorem}[section]

\newtheorem{thm}[theorem]{Theorem}

\newtheorem{remark}[theorem]{Remark}

\newtheorem{claim}[theorem]{Claim}

\newtheorem{convention}[theorem]{Convention}

\newtheorem{lemma}[theorem]{Lemma}

\newtheorem{cor}[theorem]{Corollary}

\theoremstyle{definition}
\newtheorem{definition}[theorem]{Definition}

\newcommand{\NN}{{\mathbb{N}}}
\newcommand{\RR}{{\mathbb{R}}}

\newcommand{\QQ}{{\mathbb{Q}}}
\newcommand{\ZZ}{{\mathbb{Z}}}
\newcommand{\sub}{\subseteq}
\newcommand{\sN}[1]{_{#1\in \omega}}
\newcommand{\uhr}[1]{\! \upharpoonright_{#1}}

\newcommand{\bi}{\begin{itemize}}
\newcommand{\ei}{\end{itemize}}
\newcommand{\bc}{\begin{center}}
\newcommand{\ec}{\end{center}}

\newcommand{\ES}{\emptyset}

\newcommand{\ex}{\exists}
\newcommand{\fa}{\forall}

\newcommand{\la}{\langle}
\newcommand{\ra}{\rangle}

\newcommand{\n}{\noindent}

\newcommand{\sss}{\sigma}
\newcommand{\aaa}{\alpha}

\newcommand{\lland}{\, \land \, }

\newcommand \seq[1]{{\left\langle{#1}\right\rangle}}

\newcommand\+[1]{\mathcal{#1}}

\newcommand{\ol}{\overline}

\newcommand{\lra}{\leftrightarrow}
\newcommand{\LR}{\Leftrightarrow}

\newcommand{\dom}{\ensuremath{\mathrm{dom}}}

\newcommand{\Sinf}{S_\infty}
\renewcommand{\hat}{\widehat}
 \newcommand{\pp}[3]{\prod_{#2 \in #3}   {x_#2}^{{#1}_#2}} 
 \newcommand{\si}[1]{ \{ #1 \}}

\begin{document}

%
 
%\title{Detecting isomorphism of closed subgroups of $S_\infty$ in Borel classes}
\title{The complexity of topological group isomorphism}
  \author{Alexander S.\  Kechris,  Andr\'e Nies and Katrin Tent}

\thanks{The first author was partially supported by NSF grant DMS 1464475. The second author was partially supported by the Marsden fund of New Zealand. The third author was supported by Sonderforschungsbereich 878 at Universit\"at M\"unster.}

\noindent  \address{A.\ S.\ Kechris, Department of Mathematics, Caltech, Pasadena CA 91125, \texttt{kechris@caltech.edu}}
\address{
A.\  Nies, Department of Computer Science,  Private Bag 92019, The University of Auckland, \texttt{andre@cs.auckland.ac.nz}}

\address{K.\ Tent, Mathematisches Institut, Einsteinstrasse 62, Universit\"at M\"unster, 48149 M\"unster}
\maketitle

\begin{abstract} We study the complexity of the topological  isomorphism relation for various classes of   closed subgroups of the group of permutations of the natural numbers. We use the setting of Borel reducibility between equivalence relations on Borel  spaces. For profinite,  locally compact, and Roelcke precompact groups, we show that  the  complexity is the same as the one of countable graph isomorphism. For oligomorphic groups, we merely establish  this as an upper bound. 

%Nies showed that the complexity of isomorphism for profinite separable groups is Borel equivalent to graph isomorphism. We show the analogous result for  the complexity of isomorphism  of separable locally  compact totally disconnected groups. 

%By van Dantzig's theorem each such group has a neighbourhood basis of the identity consisting of open subgroups. 
  \end{abstract}

 \section{Introduction}
 Let $\Sinf$ denote the Polish group  of permutations of $\omega$. It is well-known that the closed subgroups of $\Sinf$ (or equivalently, the non-Archimedean Polish groups) are,  up to topological  group isomorphism,  the automorphism groups of countable structures. Algebra or  model theory can sometimes be used to understand natural classes of closed subgroups of $\Sinf$.  Firstly, the separable profinite groups are precisely the Galois groups of Galois extensions of countable fields. For a second example, consider   the oligomorphic groups, namely the closed subgroups of $\Sinf$  such that for each $n$ there are only finitely many $n$-orbits. These groups   are precisely   the automorphism groups of  $\omega$-categorical structures with domain the natural numbers. Under this correspondence,  topological isomorphism turns into bi--interpretability of the structures  by a result in  Ahlbrandt and Ziegler~\cite{Ahlbrandt.Ziegler:86}  going back  to unpublished work of  Coquand.

 The closed subgroups of $\Sinf$ form the points of a standard Borel space.  Our main goal is to  determine  the complexity of the topological isomorphism relation for various classes of closed subgroups of $\Sinf$ within  the setting of Borel reducibility between equivalence relations. See \cite{Gao:09} for background on this setting. 
 
 A leading  question about  an equivalence relation $E$ on a standard Borel space $X$ is whether $E$ is classifiable by countable structures. This   means that one can in a Borel way assign to $x \in X$ a countable structure $M_x$  in a fixed countable language so that $xEy \LR M_x \cong M_y$. Among the closed subgroups of $\Sinf$, we will consider the Borel classes of compact (i.e., profinite) groups, locally compact groups, and oligomorphic groups.  We will  include the class of Roelcke precompact groups, which generalise both the compact and the oligomorphic groups.

 We introduce a general criterion on a class of closed subgroups of $\Sinf$  to show that  each of the classes above has an  isomorphism relation that is classifiable by countable structures.  Our proof that the criterion works has  two different versions. The first version is   on the descriptive set theoretical side: we Borel reduce the  isomorphism relation to conjugacy of closed subgroups of $\Sinf$, which implies classifiability by countable structures using a   result  of Becker and Kechris~\cite[Thm.\ 2.7.3]{Becker.Kechris:96}.  The second version  is on   the model theoretic side:  from a  group  $G$ in the class we directly construct a countable structure $M_G$ in a fixed finite language  so that topological isomorphism of two groups is equivalent to isomorphism of the associated structures.  
 
Independently from us, Rosendal and Zielinski \cite[Prop.\ 10 and 11]{Rosendal.Zielinski:arXiv}  established classifiability by countable structures for the isomorphism relation in the four classes above,    and posted their result on arXiv in Oct.\ 2016. Their methods are   different from ours: they obtain the results as corollaries to their theory of classification by  compact metric structures under homeomorphism. 
 
 We conversely provide a Borel reduction of  graph   isomorphism to isomorphism of profinite groups, using an extension of an argument by Mekler \cite{Mekler:81} within the framework of topological groups. In fact, for   $p $ an odd prime, the class of exponent $p$, nilpotent of class 2,  profinite groups  suffices. 
 
 For isomorphism of oligomorphic groups, it is clear that the identity on $\mathbb R$ is a lower bound (e.g.\ using Henson digraphs); we leave open the question whether this lower bound can be improved.  
 
 Recent work with Schlicht shows that   the bi-interpretability relation  for $\omega$-categorical structures is Borel.  Since graph isomorphism is not Borel, this  upper bound   for the  isomorphism relation of the corresponding automorphism groups is not sharp.  Also note that bi-interpretability of structures is equivalent to bi-interpretability of their theories (suitably defined). If the signature is finite then the Borel equivalence relation of bi-interpretability of theories over that signature   has countable equivalence classes.

Using Lemma~\ref{lem:dense sequence} below,  it is not hard to verify that the  isomorphism relation  for general closed subgroups of $\Sinf$ is analytic. It is unknown what the exact complexity of this relation is in terms of Borel reducibility.  By our result involving profinite groups, graph isomorphism is a lower bound.

 \section{Preliminaries} \label{s:prelim}
% one  uses the Wijsman topology which has as a Borel structure the Effros space of $\Sinf$. Now show that to be a subgroup is a closed condition.
\subsection*{Effros structure of  a Polish space} Given a Polish space $X$, let $\+ F(X)$ denote  the set of closed subsets of $X$. The \emph{Effros structure} on $X$  is the Borel space consisting of  $\+ F(X)$  together with  the $\sigma$-algebra   generated by the sets \bc $\+ C_U = \{ D \in \+ F(X) \colon D \cap U \neq \ES\}$,  \ec for open $U \sub X$. Clearly it suffices to take all the sets $U$ in a countable basis $\seq{U_i}\sN i$ of $X$.  The  inclusion relation  on $\+ F(X)$ is Borel because for $C, D \in 
\+ F(X)$ we have $C \sub D \lra \fa i \in \NN \,  [ C \cap U_i \neq \ES \to D \cap U_i \neq \ES]$.

The following fact will be used frequently. 
  \begin{lemma}[see  \cite{Kechris:95}, Thm.\ 12.13] \label{lem:dense sequence} Given a   Polish space $X$, there is a Borel map $f: \+ F(X)\longrightarrow X^\omega$ such that for a non-empty set  $G \in \+  F(X)$, the image $f(G)$ is a sequence $(p^G_i) \sN i$  in $X^\omega$ that is dense in $G$.   \end{lemma}
\subsection*{The Effros structure of $\Sinf$} For a Polish group $G$, we have a Borel action $G \curvearrowright \+ F(G)$ given by left translation.  In this paper we will only consider the case that $G= \Sinf$.
  In the following $\sss, \tau, \rho$ will denote injective maps on  initial segments of the integers, that is, on tuples of integers without repetitions. Let $[\sss]$  denote the  set of permutations extending~$\sss$:  \bc $\+ [\sss] = \{ f \in \Sinf \colon \sigma \prec f\}$  \ecÊ(this is often denoted  $\+ N_\sss$ in the literature).  The sets $[\sss]$ form a base for the topology of pointwise convergence of $\Sinf$.  For $f \in \Sinf$ let $f \uhr n$ be the initial segment of $f$ of length $n$. Note that the $[f\uhr n]$ form a basis of neighbourhoods of $f$. 
%For a string  $\sss$ let $\sss^{-1}$ denote the inverse, as far as it is defined on an initial segment of $\omega$. For instance, the inverse of $  (5,6, 1, 0 \ra$ is $(3 ,2 \ra$.
Given  $\sss, \sigma'$ let $\sigma' \circ \sigma$ be the composition as far as it is defined;  for instance, $(7,4,3,1,0 ) \circ (3,4 ,6 ) = (1,0)$. Similarly, let $\sss^{-1}$ be the inverse of $\sss$ as far as it is defined.

\begin{definition} For $n \ge 0$, let  $\tau_n$ denote  the  function $\tau$ defined on $\{0,\ldots, n\}$ such that $\tau(i) = i$ for each $i \le n$. \end{definition} 
\begin{definition} For $P \in \+ F(\Sinf)$,  by $T_P$ we denote the tree describing $P$ as a closed set in the sense that $[T_P] \cap \Sinf = P$. Note that  $T_P = \{ \sss \colon \, P \in \+ C_{[\sss]}\}$. \end{definition}

\begin{lemma}  \label{lem:mult} The   relation $\{(A, B, C ) \colon AB \sub C\}$ on $\+ F(\Sinf)$ is Borel. \end{lemma}
\begin{proof} $AB \sub C$ is equivalent to  the Borel condition

\bc $\fa \beta \in T_B \fa \alpha \in T_A  \, [  |\aaa| > \max \beta \to  \aaa \circ \beta \in T_C]$. \ec 
For the nontrivial implication, suppose the condition holds. Given $f \in A, g\in B$ and $n \in \NN$, let $\beta = g \uhr n$, and $\alpha = f \uhr {1+ \max \beta}$. Since $\alpha \circ \beta \in T_C$, the neighbourhood $[f\circ g \uhr n]$ intersects $C$. As $C$ is closed and  $n$ was arbitrary, we conclude that $f \circ g \in C$.
\end{proof}

\subsection*{The Borel space of non-Archimedean groups} 
\begin{lemma} The closed subgroups of $\Sinf$ form a Borel set  $\+ U(\Sinf)$ in $\+ F(\Sinf)$. \end{lemma} 

\begin{proof} $D \in \+ F(\Sinf)$ is a subgroup iff     the following three conditions hold:

\bi \item $D \in \+ C_{[(0, 1, \ldots, n-1 )]}$  for each $n$
\item $D \in \+ C_{[\sss]}  \to D \in \+ C_{[\sss^{-1}]}$  for all $\sss$
\item $D \in \+ C_{[\sss]}  \cap C_{[\tau]}  \to D \in \+ C_{[\tau \circ \sss]}$ for all     $\sss, \tau$. \ei
It now suffices to observe  that all three conditions are Borel. 
\end{proof}

Note that $\+ U(\Sinf)$ is a standard Borel space.
 The statement of the  lemma  actually holds for each Polish group in place of $\Sinf$.
  
%  ; in the general case one  uses e.g\  the Wijsman topology, which generates the Effros structure of $G$.  

\subsection*{Locally compact closed subgroups of $S_\infty$} These groups are exactly the (separable)  totally disconnected  locally compact   groups. The  class of such groups has been widely studied. A set $D \in \+ F(\Sinf)$ is   compact iff the tree $T_D= \{ \sss \colon D\in \+ C_{[\sss]}\}$ is finite at each level. 
A closed subgroup $G$ of $ \Sinf$ is locally compact iff some point in $G$ has a compact neighbourhood. Equivalently,  there is $\tau $ such that $G \in  \+ C_{[\tau]}$ and the tree  $\{ \sss \succeq \tau \colon G\in \+ C_{[\sss]}\}$ is finite at each level.   Thus, compactness and local compactness of subgroups are   Borel conditions in $\+ F(\Sinf)$.

\subsection*{The canonical structure for a  closed subgroup of $\Sinf$} Given $G \in \+ U(\Sinf)$ we can in a Borel   way obtain a countable structure $M_G$ in a countable signature such that $G  \cong \mathrm{Aut}(M_G)$. For each $n$,  order the $n$-tuples lexicographically. Let $\seq {\ol a_i}_{i< k_n}$, where $k_n \le \omega$,  be  the ascending  list of the $n$-tuples that are  least in their orbits. The signature has $n$-ary predicate symbols $P^n_i$ for $i< k_n$, where the symbol $P^n_i$ is  interpreted as the $G$-orbit of $\ol a_i$.

%\begin{prop} The relation of continuous isomorphism on  $\+ U(\Sinf)$ is analytical. \end{prop}
% Easiest way to write this may be using the representation  of closed subgroups of $\Sinf$ in Section~\ref{s:RT} where the inverse of an element is included in its code.

   \section{A sufficient criterion for classifiability of  \\ topological isomorphism   by countable structures}

 We show that the isomorphism relation on a  Borel class $\+ V$ of subgroups   of $S_\infty$  that is invariant under conjugacy   is classifiable by countable structures,   provided that  we can canonically assign to each $G \in \+ V$ a countable base of neighbourhoods of $1$ that are open subgroups. 
   For instance,  in the case of locally compact groups $G$, we can  take as $\+ N_G$ the compact open subgroups of~$G$. 
 This relies on    the classic result  of  
 van Dantzig: if a  totally disconnected  group $G$  is locally compact, it has a compact open subgroup  $U$. Note that $U$ is profinite, so in fact   the identity element of $G$  has a  basis of neighbourhoods consisting  of compact open subgroups.

\subsection{The sufficient criterion}
 %%%%
 \begin{thm} \label{thm:general criterion} Let $\+ V $ be a Borel set of subgroups of $S_\infty$ that is  closed under conjugation. 
  
  Suppose that  for each $G \in \+ V$ we have a  countably infinite  set $\+ N_G$ of open subgroups of $G$ that forms a neighbourhood basis of $1$. 
  Suppose further that the relation 
  \[\+ T = \{ (G, U) \colon  \, G \in \+ V\mbox{ and } U \in \+ N_G \} \]
  is Borel, as well as isomorphism   invariant in the sense that \bc $ \phi \colon G \cong H$ implies $V \in \+ N_G \lra \phi(V) \in \+ N_H $. \ec
  Then:
  \begin{itemize}  \item[(i)] The isomorphism relation on $\+ V$  is Borel reducible to conjugacy  of closed subgroups of $S_\infty$. (Moreover, a  conjugating permutation can be obtained in a Borel way from an isomorphism, and vice versa.) 
  \item[(ii)] In particular, the  isomorphism relation on $\+ V$ is classifiable by countable structures.  \end{itemize}
  \end{thm}
  %%%%%%%%%%%%%%%%
  
  \begin{proof}   
We assign in a Borel way to each $G \in \+ V$ a closed subgroup $\hat G$ of $S_\infty$ so that  $ G \cong H$ iff $\hat G $ is conjugate to $\hat H$.   Moreover, we ensure that  $G \cong \hat G$.

For $G \in \+ V$, let   $\+L(G)$ denote the set of left  cosets of the subgroups in $\+ N_G$.  The relation
  \[\+ S = \{ ((G, B ) ,  U ) \colon  \, G \in \+ V , U \in \+ N_G \mbox{ and } \ex i \, B= p_i^G U \} \]
 is Borel and has countable sections on the second component because each $\+ N_G$ is countable. By  a result of Lusin-Novikov,  in the version of  \cite[18.10]{Kechris:95},  the  projection of $\+ S$ onto the first component is Borel. This projection  equals $ \{ (G, B ) \colon  \, G \in \+ V \land B \in \+ L(G) \}$.  Note also that $ \+ L(G) $ is countably infinite.  So by the same result \cite[18.10]{Kechris:95}  there is a Borel function $F \colon \+ V \to \+ F(\Sinf)^\omega$ taking a group $G \in \+ V$ to a bijection $ \eta_G\colon \omega \to  \+ L(G)$.

 The group $G$ acts by left translation on $\+ L(G)$. So  via the  bijection $\eta_G$, the  left  action of an element  $g$ on $\+ L(G)$ corresponds    to a permutation $\Theta_G(g) \in \Sinf$.    We let $\hat G$ be the range of  $ \Theta_G$ and verify that the  map $G \to \hat G$ is Borel and has the  desired property. 
 We will omit the subscript $G$ for now and simply write $\Theta$.

\begin{claim} The map $\Theta$ is a topological group isomorphism $G \cong \Theta(G)$.  \end{claim}
Clearly $\Theta$ is a homomorphism of groups. To show that    $\Theta$ is 1-1 suppose that $g \neq 1$ and pick $V \in \+ N_G$ such that $g \not \in V$. Then $gV \neq V$, hence $\Theta(g) (\eta^{-1}_G(V)) \neq \eta^{-1}_G(V)$, so that $\Theta(g) \neq 1$.

%:
For  continuity of $\Theta^{-1}$ at $1$, suppose $U\in \+ N_G$, and let  $n= \eta_G(U)$. If $\Theta(g)(n) = n$ then $g \in U$.

For continuity of $\Theta$ at $1$, consider $C \in \+ L(G)$ and let $n= \eta_G(C)$. If $gC = C$ then $\Theta(g)(n) = n$. So it suffices  to find a neighbourhood $W$ of $1$ in $G$ such that $g \in W \to gC= C$.  Choose $r \in G$ and $U \in \+ N_G$ such that $C = rU$. The desired neighbourhood $W$ is $\{ g \colon \, r^{-1} g r \in U \}  =  rU r^{-1}$. This shows the claim.

 Since $\Theta(G)$ is Polish it is  a $G_\delta$ subgroup of  $\Sinf$, and hence closed by the Baire category theorem \cite[1.2.1]{Becker.Kechris:96}. 

   \begin{claim} The map $L \colon \, \+ V \to \+ F(S_\infty)$ sending $G $ to $  \hat G$ is Borel. \end{claim} 
We use that  the action by left translation   $G \curvearrowright  \+ F(S_\infty)$  is  Borel,  and the assignment  $G \to \eta_G$ is Borel. In  the notation of  Section~\ref{s:prelim}, we have to show that the preimage of $\+ L_{[\sss]}$ under $L$ is Borel for each tuple  $\sss$ without repetitions. This preimage equals $\{ G \colon \hat G \cap [\sss]\neq \ES\}$, which in turn  equals $\{ G \colon \exists g \in  G \, \Theta_G(g) \in [\sss]\}$.   For each $G$ the set of such $g$ is open in $G$ because $\Theta_G$ is continuous. So if this set is nonempty it contains $p_i^G$ for some $i\in \omega$, where $p_i^G$ is given by Lemma~\ref{lem:dense sequence}.  
 To say that $\Theta_G(p_i^G) \in [\sss]$ means that \bc  $\fa n, k \,  [ \sss(n) =k  \to  p_i^G  \eta_G(n)  = \eta_G(k) ]$,  \ec 
 which is a Borel property of $G$. The  preimage of $\+ L_{[\sss]}$ under $L$ equals the union of these sets over all $i$, which is therefore Borel.  This shows the claim.
 %Since all the groups in $\+ N_G$ are open in $G$,

Suppose now that $G , H \in \+ V$. First let $\phi \colon G \to H$ be a topological group isomorphism. By our hypothesis that  the relation  $\+ T$ is isomorphism invariant, $\phi$ induces a bijection $\Phi \colon \+ L(G) \to \+ L(H)$ via $\Phi(rU) = \phi(r) \phi(U)$ for $U \in \+ N_G$. Then $\Theta_G(G)$ is conjugate to $\Theta_H(H)$ via the permutation $\alpha = \eta_H^{-1} \circ \Phi \circ \eta_G$, because for $p = \Theta_G(g)$ we have $\alpha\circ p \circ \alpha^{-1} = \Theta_H(\phi(g))$. 

Next suppose that  $\Theta_G(G)^\alpha = \Theta_H(H)$ for $\alpha \in \Sinf$. Then  $\phi= \Theta_H^{-1} \circ \alpha \circ \Theta_G$ is a topological group isomorphism of $G$ and $H$.

%
%(In more detail, let $n = \eta_H(sV)$ for $s\in H, V \in \+ N_H$. Let $rU = \phi^{-1}(sV)$. Then $\alpha^{-1}(n)= rU , \Theta_G(g)(rU) = grU$, and where $\eta_G(grU) =m$ we have $\eta^{-1}_H\alpha(m) =\phi(g)\phi(r) \phi(U)$. On the other hand $\Theta_H(\phi(g))$ applied to $n$ also yields $k$ with $\eta_H^{-1}(k) = \phi(g) sV= \phi(g)\phi(r) \phi(U)$.)

%Next suppose that  $\Theta_G(G)^\alpha = \Theta_H(H)$ for $\alpha \in \Sinf$. Then  $\phi= \Theta_H^{-1} \circ \alpha \circ \Theta_G$ is a topological group isomorphism of $G$ and $H$.
%
%
% 
% 
Note that one can  obtain $\alpha $ from $\phi$ in a Borel way, and vice versa.  This shows (i).
For (ii), recall that any $\Sinf$-orbit equivalence relation is classifiable by countable structures by a result of Becker and Kechris~\cite[Thm.\ 2.7.3]{Becker.Kechris:96}; also see   \cite[Thm.\ 3.6.1]{Gao:09}. 
  \end{proof}
 
 We now apply the criterion given by the foregoing theorem to various classes of groups. We have already mentioned how to  obtain the case of locally compact groups. 
 
  A closed subgroup $G$ of $S_\infty$ is  called \emph{oligomorphic} if  for each  $n$ there are only finitely many $n$-orbits for the natural action of $G$ on $\omega$. Clearly this property is Borel.
Each open subgroup of $G$ contains an open subgroup  $[\tau_n ] \cap G$ for some $n$, where $\tau_n$ denotes the identity   tuple  $(0,1, \ldots, n)$. Since $[\tau_n ] \cap G$ has only finitely many cosets in $G$, there are only countably many open subgroups. 
   
A topological  group $G$ is called  \emph{Roelcke precompact} if for each   neighbourhood $U$ of $1$, there is a finite set $F$  such that $UFU = G$. If  $G$ is a closed subgroup of $S_\infty$,  we may assume that  $U$ is an open subgroup of $G$, in which case the defining condition states that the double coset equivalence relation $\sim_U $ given by $x \sim_U y \lra  UxU = UyU$ has only finitely many equivalence classes.  

 Roelcke pre-compactness is   a Borel property of subgroups of $S_\infty$: given an open subgroup $U$, if $F$ as above exists we can choose the elements of $F$ among the dense sequence $p_i^G$ obtained in a Borel way as in Lemma~\ref{lem:dense sequence}. It suffices to require that $F$ exists for each  $U = [\tau_n] \cap G$.

 We note that a  Roelcke precompact non-Archimedean group $G$  has only  countably  many open subgroups: if $U$ is an open subgroup, then each  $\sim_U$ class is a finite union of $\sim_{[\tau_n] \cap G}$ classes for an appropriate $n$.    Now  $U$ itself is the equivalence class of $1_G$ under $\sim_U$. So there are only countably many options for $U$.

 \begin{cor} Let $\+ V $ be a Borel set of subgroups of $S_\infty$ that is  closed under conjugation.  Suppose that every group in $\+ V$ has only countably many  open subgroups. Then the isomorphism problem for $\+ V$   is classifiable by countable structures. 
 
 In particular, this is  the case for the oligomorphic and Roelcke precompact groups.
 \end{cor}
 \begin{proof} Let $\+ N_G$ be the set of all open subgroups of $G \in \+ V$. The relation $\+ T$  in Thm.\ \ref{thm:general criterion} is Borel because a closed group  $H\le G$ is open in $G$ iff $G \cap [\tau_n] \sub H$ for some $n$, where as before $\tau_n$ is the identity tuple of length $n+1$.  Next, $G \cap [\tau_n] \sub H$ is equivalent to $\fa \rho \succeq \tau_n  (  [\rho] \cap G \neq \ES \to [\rho] \cap H \neq \ES)$.  $\+ T$ is clearly isomorphism invariant, so all the hypotheses of the theorem hold. 
 \end{proof}

\subsection{Remarks}  Roelcke precompactness generalises both compactness and being oligomorphic.  For the second statement, note that by  Tsankov~\cite{Tsankov:12}, $G$ is Roelcke precompact iff  $G $ is the inverse limit of a diagram \bc $\ldots \to G_3 \to G_2 \to G_1$,  \ec where each $G_i$ is an oligomorphic permutation group on some countable set.    
As pointed out in~\cite[before Prop.\ 2.2]{Tsankov:12} the only Roelcke precompact and locally compact Polish  groups $G$ are the compact ones.
%
%: Suppose  $U$ is a compact open subgroup of $G$ by van Dantzig's theorem. For  any $a \in G$ the double coset $UaU$ is compact as the composition of  two compact sets. $G$ is the finite union of such double cosets.

  %Suppose $U $ is an open subgroup of $G$. 
 We note that besides   the Roelcke precompact groups, there are further non-Archimedean groups with only countably  many open subgroups, for instance many (discrete) countable groups.  For an uncountable example, consider the locally compact group $\mathtt{PSL_2}(\QQ_p)$: by a result of Tits,  each open  subgroup is  either compact, or the whole group. We thank Pierre-Emmanuel Caprace for pointing out this example. 
 
 %We may choose $i$ and an open subgroup $V$ of $ G_i$ such that $\wt V = q_i^{-1} (V) \le U$. There are only countably many possibilities for such a choice. 

%By \cite[2.2iv]{Tsankov:12},  

%$U$ is    Roelcke precompact itself, so it is a finite union of double cosets  $\wt V r \wt V$. Since $\wt V$ is open it has only countably many left cosets in $G$, \end{proof}

\subsection{A direct construction of  structures for  Theorem~\ref{thm:general criterion}(ii)}
To prove Theorem~\ref{thm:general criterion}(ii)  we used  that any $\Sinf$-orbit equivalence relation $E$  is classifiable by countable structures~\cite[Thm.\ 2.7.3]{Becker.Kechris:96}. This result actually gives an $\Sinf$ reduction of $E$ to the logic action for an infinitary language. On the downside, the proof is quite indirect, making use of the fact that the action  by  left translation of $\Sinf$ on $\+ F(\Sinf)^\omega$ is universal,  and then encoding sequences of closed  sets by a countable structure based on the sequence of  corresponding  trees. 

We now  give a direct construction of  the structures  in Theorem~\ref{thm:general criterion}(ii).
Recall that   the structures for a fixed countable relational   language $L = (R_i)_{i \in I}$ form a Polish space $X_L = \prod_I  \+ P(\omega^{n_i})$ (where $n_i$ is the arity of $R_i$).
 We now define a finite relational language $L$ and a Borel function that assigns to  $G \in \+ V$ a countable $L$-structure $M_G$ in a  such a way that groups $G,H \in \+ V$ are topologically isomorphic iff $M_G \cong M_H$ as $L$-structures.
   
Given a group $G$, the universe of the structure $M_G$ consists of the left and right cosets of subgroups in $\+ N_G$. Since $ \+ N_G$ is countable and each group in $\+N_G$ is open in $G$, this is a countable set.  The language $L$ consists    
of a ternary relation $R$ interpreted as $AB\subseteq C$.
 Again by the result of Lusin-Novikov as  in  \cite[18.10]{Kechris:95}
 there is a Borel function taking a group $G \in \+ V$ to a bijection $ \nu_G\colon \+   \omega \to \dom (M_G)$. So we may identify the elements of $M_G$ with natural numbers.    

\begin{remark} \label{rem:defs} {\rm Being a subgroup is first-order definable in $M_G$ because a coset $A$ is a subgroup  of $G$ if and only if $AA \sub A$.  We can also express in the language of $M_G$ that  a subgroup $A$ is contained in a subgroup $B$,  using the first-order  formula $AB \sub B$. We can say that   a coset $A$ is a left coset of a subgroup $U$ by expressing that $U$ is the maximum subgroup     with the property that $AU \sub A$; similarly for $A$ being a right coset of $U$. Finally,  we can express that $A\sub B$ for arbitrary cosets $A,B$ by expressing that  $AU \sub B$ for the smallest  subgroup $U$. } \end{remark}

%
%Similarly, being a left (right, resp.) coset of a subgroup is definable, as is inclusion of cosets (of any kind).
%
%Wir koennen sagen, dass A eine Untergruppe ist:
%
%    $ AA\subset A$
%
%Wir koennen sagen, dass eine Untergruppe A in einer Untergruppe B
%enthalten ist:
%

%  A, B sind Untergruppen und $AB\subset B$
%
%
%Dann koennen wir sagen, dass A Nebenklasse einer Untergruppe U ist:
%
%      U ist maximale Untergruppe mit $AU\subset A$
%
%
%Wir koennen fuer beliebige Nebenklassen A,B ausdruecken, dass $A\subset B$:
%
%Ist U eine Untergruppe mit $AU\subset A$, dann ist $ AU\subset B$.
%%%%%%%%
 \begin{claim} The function $G \mapsto M_G$ (for $G \in \+ V$)  is   Borel.  \end{claim}

\begin{proof}
Given a triple of numbers $\ol a$,  we need  to show that the preimage of the set of structures that satisfy $R \ol a$, that is,   $\{ G \colon \, M_G \models R \ol a\}$,   is Borel on $\+ F (\Sinf)$. 
 The preimage of the structures where the product of the   cosets denoted by $a_0, a_1$ is co	ntained in the   coset denoted by $a_2$ is $\{G \in \+ V \colon \, \nu_G(a_0) \nu_G(a_1)  \sub \nu_G(a_2)\}$.  As observed  in    Lemma~\ref{lem:mult}, $AB \sub C$  is a Borel relation on $\+ F (\Sinf)$. Since the assignment $G \mapsto \nu_G$ is also Borel, this preimage is Borel.     
\end{proof}
First suppose that  $G \cong H$ via $\theta$. By the  hypothesis of isomorphism invariance for   the relation $\+ T$ defined  in Theorem~\ref{thm:general criterion},  $M_G \cong M_H$ via $\nu_H^{-1}  \circ \Theta \circ \nu_G$, where $\Theta$ is the  map on cosets induced by $\theta$. 

Next we show that  $M_G \cong  M_H$  implies that $G \cong H$.
To $g \in G$ we associate the pair  $L_g, R_g$, where $L_g$ is the set  of left cosets in $M_G$ containing $g$,  and $R_g$ the set of right cosets containing $g$. Both sets are neighbourhood bases of $g$ consisting of open sets. So we have the following properties of $L=L_g$ and $R = R_g$: 

\begin{itemize} \item[(1)] $L$ and $R$ are downward directed under inclusion.  

\item[(2)]
 For each $A\in L$ and $B\in R$   there is  $C \in L$   such that $C  \sub A\cap B$. 
 
 \item[(3)]
Each group  $U \in \+ N_G$ has a left coset in $L$ and a right coset in $R$ (these cosets are  necessarily unique). 
 \end{itemize}
 
Suppose now we have a pair $L,R$ with the properties above. We construct an element $g$ of $G$ such that $L= L_g$ and $R= R_g$. 

Let $U_0 \in  \+ N_G$. Using that $\+ N_G$ is a neighbourhood basis, for each $n>0$,   let $U_n \in \+ N_G$ be such that $U_n \sub   [\tau_n] \cap U_{n-1}$,  where  as above $\tau_n$ is the identity string of length $n+1$.  Define $g(n) = r_n(n)$ where  $r_n \in G $  and $r_nU_n $ is the left  coset of $U_n$ in $L$. Also define $g^*(n) = s_n^{-1}(n)$ where $U_n s_n $ is the right  coset of $U_n$ in $R$.  

 Clearly the $r_n \uhr {n+1} $ ($n \in \omega$) are compatible strings with union $g$. Similarly the $s^{-1} _n \uhr {n+1} $ are compatible strings with union $g^*$.   However, so far we don't know that the injective function  $g$ is a permutation.

\begin{claim} \label{cl:mult} $g^*$ is the inverse of $g$. In particular,  $g \in \Sinf$. \end{claim}
\begin{proof}
Note that   $r_n$ determines the first $n+1$ values of any permutation $f$ in $r_n U_n$, namely $f(i) = r_n(i) $ for each $i \le n$. Likewise, $ h \in U_n s_n$ implies  that $h^{-1} \in s_n^{-1} U_n$ and hence $h^{-1}(i) = s^{-1}_n(i) $ for each $i \le n$. 

Assume for a contradiction that, say $g(x)= y$ and $g^*(y) \neq x$. Let $n = \max(x,y)$. Then  $r_nU_n \cap U_n s_n = \ES$, contrary to (2) above:  if $f \in r_nU_n \cap U_n s_n$ then by the compatibility and the  observation above,  $f(x) = r_n(x) = y$, while $f^{-1} (y) = s_n^{-1}(y) \neq x$. 
\end{proof}

\begin{claim} (i) $g \in G$. (ii) $L_g = L$ and $R_g = R$. \end{claim}
\begin{proof} (i)  follows because $g \uhr {n+1} = r_n \uhr {n+1}$ and  $G$ is closed.  (ii) is clear from the definition of $g$ and $g^* = g^{-1}$. \end{proof}

\begin{claim} For  $a,b,c \in G$,

$a \circ b= c \lra \fa C \in L_c \  \ex A \in L_a \ex B \in L_b \ [AB \sub C]$.  \end{claim}
\begin{proof} The implication from left to right holds  by continuity of composition in $G$.  For the converse implication, suppose that $(a \circ b)(n) \neq  c(n)$. Let   $C \in L_c$   be the left coset of $U_n$. Then $d(n) = c(n)$ for any $d \in C$. If $A \in L_a$ and $B \in L_b$ then $a\circ b \in AB $ so that $AB \not \sub C$.  \end{proof}

Now suppose that $M_G \cong  M_H$ via $\rho$. Given $g \in G$, note that the  pair $\rho(L_g), \rho(R_g)$ has the properties (1-3) listed above using Remark~\ref{rem:defs}. So let   $\theta (g)$ be the element of $H$ obtained from this pair. Similarly,   the inverse of $\theta$ is determined by  the inverse of $\rho$.  

By Claim~\ref{cl:mult}, $\theta$ preserves the composition operation.  The identity of $G$ is the only element $g \in G$ such that $L_g = \+ N_G$. Since $\rho$ is an isomorphism of the structures, $\theta(1_G) = 1_H$,  and $\theta$ and $\theta^{-1} $ are  continuous at $1$. So $\theta $ is a topological isomorphism.

  \section{Hardness result for isomorphism of profinite groups} 
  Graph isomorphism  is $\le_B$-complete for $\Sinf$-orbit equivalence relations. 
We now consider the converse problem of Borel reducing graph isomorphism to isomorphism on a Borel class of nonarchimedean groups.  
  We first consider the  case of discrete    groups.
Essentially by a result of  Mekler~\cite[Section 2]{Mekler:81} discussed in more detail below, graph isomorphism is Borel reducible to isomorphism of countable groups. 
 Given a discrete group $G$ with domain $\omega$, the left translation action of $G$ on itself induces a topological isomorphism  of $G$ with a discrete subgroup of $\Sinf$. Hence graph isomorphism is Borel reducible to isomorphism   of discrete, and hence of   locally compact,   subgroups of $S_\infty$.% via the left translation action.  

  We now show a similar hardness result for  the compact non-Archimedean groups (equivalently, the separable profinite groups).
  Recall that a  group $G$ is step 2  nilpotent  (nilpotent-2 for short) if it satisfies the law $[[x,y],z]=1$. Equivalently, the commutator subgroup  is contained  in the center. For a prime $p$, the group of unitriangular matrices $$\mathrm{UT}^3_3(\ZZ / p \ZZ) = \big \{ \left(\begin{matrix} 1 & a  & c \\  0  &  1 & b \\ 0 &  0 & 1
\end{matrix}\right) \colon \, a,b,c \in \ZZ / p \ZZ\big \}$$ is an example of a nilpotent-2 group of exponent~$p$.

\subsection{Completion} \label{ss:completion}
 We need some preliminaries on the completion of a group $G$ with respect to a system of subgroups of finite index. We follow \cite[Section 3.2]{Ribes.Zalesski:00}. 
 Let  $\+ V$ be a set of normal subgroups of finite index in $G$ such that $U, V \in  \+ V $ implies that there is $W \in \+ V$  with  $W \sub U  \cap V$. We can turn $G$ into a topological group by declaring $\+ V$ a basis of neighbourhoods  of the identity. In other words, $M \sub G$ is open  if for each $x \in M$ there is $U \in \+ V$ such that $xU \sub M$. 
 
 The completion of $G$ with respect to $\+ V$ is the inverse limit  $$G_{\+ V} = \varprojlim_{U \in \+ V} G/U,$$ where $\+ V$ is ordered under inclusion and the inverse system is equipped with  the natural    maps: for $U \sub  V$, the  map  $p_{U,V} \colon G/U \to G/V $  is given by $gU \mapsto gV$. The inverse limit can be seen as a closed subgroup of the direct product $\prod_{U \in \+ V} G/U$ (where each group $G/U$ carries  the discrete topology), consisting of the functions~$\alpha $ such that $p_{U,V}(\alpha (gU)) = gV$ for each $g$.  Note that the map $g\mapsto (gU)_{U \in \+ V}$ is a continuous homomorphism $G \to G_{\+ V}$ with dense image; it is injective iff $\bigcap \+ V = \{1\}$. 
 
Note that if $\+ V$ and $\+ W$ are equivalent bases of neighbourhoods of the identity (i.e., $\forall V \in \+ V \exists W \in \+ W \,  [ W \sub V]$ and conversely) then the resulting completions $G_\+ V$ and $G_\+ W$ are homeomorphic.   If the set $\+ V$ is understood from the context, we will usually  write $\widehat G$ instead of $G_\+ V$.

\subsection{Review of Mekler's construction} Fix an odd  prime $p$. The   main construction  in   Mekler \cite[Section 2]{Mekler:81}   associates to each symmetric and irreflexive graph $A$ a nilpotent-2 exponent-$p$   group $G(A)$ in such a way that isomorphic graphs yield isomorphic groups.  In the countable case, the map $G$ is Borel  when viewed as  a  map from the  Polish space of countable graphs to the space of countable groups.

To recover $A$ from $G(A)$, Mekler uses a technical restriction on the given graphs.
\begin{definition} A symmetric and irreflexive graph is  called \emph{nice} if it has no triangles, no squares, and for each pair of distinct vertices $x,y$, there is a vertex $z$ joined to $x$ and not to $y$.   \end{definition} 
 Mekler \cite{Mekler:81} proves that   a nice graph   $A$ can be interpreted in $G(A)$ using first-order formulas  without parameters. (See \cite[Ch.\ 5]{Hodges:93} for background on interpretations.) In particular,  for nice graphs $A, B$  we have  $A \cong B$ iff $G(A) \cong G(B)$. 
Since isomorphism of nice graphs is Borel complete for $S_\infty$-orbit equivalence relations, so is isomorphism of countable nilpotent-2 exponent $p$ groups.  For an alternative    write-up of Mekler's construction see \cite[A.3]{Hodges:93}.

In the following all graphs will be  symmetric, irreflexive, and have  domain $\omega$. Such a graph is thus given by its set of edges   $A \sub \{ (r, s ) \colon \, r < s \}$.  We write $r A s$ (or   simply $rs$ if $A$ is understood)  for $(r, s ) \in A$. 

Let $F$ be the free nilpotent-2 exponent-$p$ group on free generators $x_0, x_1, \ldots$. For $r \neq s$ we write $$ x_{r,s} = [x_r, x_s].$$ As noted in \cite{Mekler:81}, the centre $Z(F)$ of $F$ is an elementary abelian $p$-group (so an $\mathbb{F}_p$ vector space) with basis $x_{r,s}$ for $r< s$. Given a graph $A$, Mekler   sets 
\[ G(A) = F/ \la x_{r,s} \colon \, r A s \ra. \]
In particular $F = G(\ES)$. The centre  $Z= Z(G(A))$ is an abelian group of exponent $p$ freely generated by the $x_{r,s}$ such that  $\lnot r A s$.  Also $G(A)/Z $ is 
an abelian group of exponent $p$ freely generated by the $Z x_i$. (Intuitively,  when  defining $G(A)$ as a quotient of $F$,   exactly the commutators $x_{r,s}$ such that  $r A s$ are deleted. We make sure that no    vertices are deleted.)  

\begin{lemma}[Normal form for $G(A)$, \cite{Mekler:81,Hodges:93}] \label{lem: NF G}Every element $c$ of $Z$ can be uniquely written  in the form $ \prod_{(r, s ) \in L} x_{r,s}^{\beta_{r,s}}$ where    $L \sub \omega \times \omega$ is a finite set of pairs $(r, s )$ with $r<s$ and  $\lnot rAs$, and $0 < \beta_{r,s}  < p$. 

Every element of $G(A)$ can be uniquely written  in the form $c \cdot v$  where  $c \in Z$, and $v=  \prod_{i \in D} x_i^{\alpha_i}$, for $D \sub \omega$   finite and $0 < \alpha_i < p$.  (The product  $\prod_{i \in D} x_i^{\alpha_i}$ is interpreted    along   the indices  in ascending order.)
 
\end{lemma}

\subsection{Hardness result for  isomorphism of  profinite groups} The following first appeared in pre\-print form in  \cite{Nies:16}.

\begin{theorem} \label{thm:Mekler} Let $p \ge 3$ be a prime. Any $S_\infty$ orbit equivalence relation can be Borel reduced to isomorphism between profinite nilpotent-2   groups  of exponent~$p$.  
\end{theorem}

We note that isomorphism on the class of abelian compact subgroups of $S_\infty$ is not Borel--above graph isomorphism as shown in \cite{Nies:16}. So in a sense the class  of nilpotent-2 groups of fixed exponent $p$  is the smallest possible.
\begin{proof} The  proof is based on Mekler's, replacing the groups $G(A)$ he defined by their completions $\hat G(A)$  with respect to a suitable basis of neighbourhoods of the identity. 
Given  a graph $A$, let $R_n$ be the normal subgroup of $G(A)$   generated by the $x_i$, $i \ge  n$.    Note that $G(A)/R_n$ is  a finitely generated  nilpotent torsion group, and hence finite.   Let    $  \widehat G(A)$ be  the   completion   of  $G(A)$    with respect to  the   set  $\+ V = \{R_n \colon \, n \in \omega\}$     (see Subsection~\ref{ss:completion}). By Lemma~\ref{lem: NF G} we have $\bigcap_n R_n= \{1\}$, so $G(A)$ embeds   into $\hat G(A)$. 
 
In  set theory one inductively defines $0 = \ES$ and $n = \{0, \ldots, n-1\}$ to obtain the  natural numbers; this will  save on notation here. A set of coset representatives for $G(A)/R_n$ is given by the $c\cdot v$  as in  Lemma~\ref{lem: NF G}, where  $D \sub n$ and $E \sub n \times n$.  The completion $\hat G(A)$  of $G(A)$ with respect to the $R_n$ consists of the maps $\rho  \in \prod_n G(A)/R_n$ such that $\rho (g R_{n+1}) = gR_n$ for each $n \in \omega$ and $g \in G(A)$. 

%If $g$ is a coset representative for $R_{n+1}$ then a coset representative $h$ for $gR_n$ is obtained  by removing all the factors involving $x_{r,n}$ or $x_n$.  

If $\rho (g R_{n+1}) = hR_n$ where $h = c \cdot v$ is a coset representative,  then we can define a coset representative  $c' \cdot v'$ for $gR_{n+1}$   as follows: we obtain $c'$   from $c$ by potentially appending to  $c$ factors involving  the $x_{r,n}$ for $r< n$, and  $v'$ from $v$ by potentially appending a factor    $x_n^{\alpha_n}$.   So we can view $\rho $ as given by multiplying two  formal infinite products:
%%%%%%%%%%%%
\begin{lemma}[Normal form  for $\hat G(A)$] 
Every   $c \in Z(\hat G(A))$ can be written uniquely in the form $ \prod_{(r, s ) \in L} x_{r,s}^{\beta_{r,s}}$ where    $L \sub \omega \times \omega$ is a   set of pairs $(r, s \ra$ with $r<s$, $\lnot rAs$, and $0 < \beta_{r,s}  < p$.

 Every element of $\hat G(A)$ can be written uniquely in the form $c \cdot v$, where $v =  \prod_{i \in D} x_i^{\alpha_i}$,    $c \in Z(\hat G(A))$, $D \sub \omega$, and $0 < \alpha_i < p$ (the product is taken along  ascending indices).  

 \end{lemma}

We can define   the infinite products above explicitly as limits in $\hat G(A)$. We view $G(A)$ as embedded into $\hat G(A)$. Given formal products as above, let \bc $v_m =\prod_{i \in D\cap m} x_i^{\alpha_i}$   and  $c_m = \prod_{(r, s) \in L \cap m \times m} x_{r,s}^{\beta_{rs}}$. \ec
For $k \ge n$ we have  $v_k^{-1}  v_n \in  {R_n}$ and $c_k^{-1} c_n \in {R_n}$. So   $v = \lim_n  v_n $ and $c = \lim_n c_n $ exist in $\hat G(A)$ and equal the values of the formal products as defined above. 
 
Each nilpotent-2 group satisfies the distributive law $[x,yz] = [x,y][x,z]$.  This implies that
$[x_r^\alpha, x_s^\beta] = x_{r,s}^{\alpha \beta}$. The following lemma generalises to infinite products the expression for commutators  that were obtained using these identities in   \cite[p.\ 784]{Mekler:81} (and also in \cite[proof of Lemma A.3.4]{Hodges:93}).

\begin{lemma}[Commutators] \label{lem:com} Let $D, E \sub \omega$.  The following holds in $\hat G(A)$.  \[[\pp \alpha r D,  \pp \beta s E] = \prod_{r \in D, \, s \in E, \,  r< s, \, \lnot rs} x_{r,s}^{\alpha_r \beta_s- \alpha_s \beta_r}\]\end{lemma}  

\begin{proof} Based on the  case of finite products, by continuity of the commutator operation and  using the expressions for limits above, we have 
\begin{eqnarray*} [\pp \alpha r D,  \pp \beta s E] &=& [\lim_n \pp \alpha r {D \cap n},  \lim_n \pp \beta s {E\cap n}] \\ 
&=&  \lim_n \prod_{r \in D, \, s \in E, \,  r< s< n, \, \lnot rs} x_{r,s}^{\alpha_r \beta_s- \alpha_s \beta_r} \\ 
&=& \prod_{r \in D, \, s \in E, \,  r< s, \, \lnot rs} x_{r,s}^{\alpha_r \beta_s- \alpha_s \beta_r}.\end{eqnarray*}
\end{proof}
Let	$C(g)$ denote the  centraliser of a group element $g$. The following is a direct consequence of Lemma~\ref{lem:com}.   
 \begin{lemma} \label{lem:centr} Let $  w \in \hat G(A) $. If  $0<  \gamma <p$ we have $C(  w^\gamma)= C(  w)$. \end{lemma}
Mekler's  argument employs the niceness of $A$ to show that a  copy of the set of vertices  of the given graph is first-order definable in $G(A)$. The copy of  vertex~$i$ is a certain definable equivalence class of the generator~$x_i$. 
 He provides  a first-order  interpretation $\Gamma$ without parameters such that $\Gamma(G(A)) \cong A$. We will show  that his interpretation has the same effect  in the profinite case: $\Gamma(\hat G(A)) \cong A$.

 We first summarize Mekler's interpretation $\Gamma$. Let $H$ be a group with centre  $Z(H)$. \bi \item For $a \in H$ let $\bar a$ denote the  coset $a Z(H)$.  
 \item Write $\bar a \sim \bar b$ if $C(a)= C(b)$. Let $[\bar a]$ be the $\sim$ equivalence class of $\bar a$.  (Thus, for $c \in Z(H)$ we have $[\bar c] = \{ Z(H)\}$.) \ei

Given a group $H$ we define the vertex set of the graph $\Gamma(H)$ to
be
 \[\Gamma(H) = \{ [ \bar a] \colon \, a\in H\setminus Z(H) \mbox{ and } |[\bar a]| = p-1    \mbox{ and } \] 
 \[ \ \ \ \  \ex  w\in H\setminus Z(H) \mbox{ such that } [\bar w] \neq  [ \bar  a ]\mbox{ and }[a,w]=1  
 \}\]

and the edge relation $R$ is given by  
 \[[\bar a ] R [\bar b] \mbox{ if }  [\bar a]\neq [\bar b] \mbox{ and } [a,b]=1.\]  
Note that the vertex set is interpretable and the edge relation definable in $H$.

We are ready to verify that a nice graph can be recovered from its associated profinite group via the interpretation $\Gamma$. 
  \begin{lemma} \label{lem:recover} For a nice graph $A$, we have $\Gamma(\hat G(A)) \cong A$ via $[\ol {x_i}] \mapsto i$.   \end{lemma}

We  follow the outline of  the  proof of \cite[Lemma 2.2]{Mekler:81}. The notation there will be adapted  to $\hat G(A)$ via allowing infinite products. 

\begin{convention} Henceforth,   in products of the  form $\pp \alpha i D$   we will assume that $D\sub \omega$ is nonempty and $0 < \alpha_i < p$ for each $i \in D$. We also set $\alpha_i=0$ for $i \not \in D$. Expressions and  calculations involving number exponents $\alpha$  will all be modulo $p$; e.g.\ $\alpha \neq 0$ means that $\alpha \not \equiv 0 \mod p$.  
\end{convention}

Clearly the edge relation is the only possible one, so it suffices to show that 
  \[\Gamma(\hat G(A))=\{[\bar x_i] \colon i \in \omega \}.\]
  
Now let $v \not \in Z(\hat G(A))$. We may assume that    $  v= \pp \alpha  i D$ satisfying  the conventions above.   We will show that   \bc $ [\bar v] = [\bar x_e]$ for some $e$ if and only if $|[\bar v]| = p-1 \mbox{ and } \ex \bar w  \,  [ \bar  v ] R  [\bar w]$.  \ec

 We distinguish four cases:  in Case 1 we check that $v$ satisfies the two  conditions on the right hand side, in Case 2-4 that it  fails. 
 Recall that $rs$ is short for $rA s$, i.e.\ that vertices $r,s$ are joined. 

 \subsection*{Case 1:   {\rm $D = \{ r \}$ for some $r$}} \
 
%\n \emph{Claim 1.} 
Note that we have $\bar v= \bar x_r^\alpha$ for some $\alpha \neq 0$. Suppose $\bar u \sim \bar v$ for $  u = \pp \beta k E $. Then $E = \{r\}$. For,  if $k \in E$, $k \neq r$, then by niceness of $A$ there is an  $s$ such that $rs  \lland \lnot   ks$,  so that $  x_s \in C(v) \setminus C(u)$ by  Lemma \ref{lem:com}. Thus $\bar u = \bar x_r^\beta$ with    $0< \beta < p$ and hence $|[\bar v]|=p-1$.

For the second condition, by niceness of $A $ we   pick $i \neq r$ such that $ir$, and let $\bar w = \bar x_i$.
%%%%%%%%%%%%%%%%
   \subsection*{Case 2: {\rm   $D = \{r, s\}$ for some $r,s$  such that   $rAs$ (and hence $r \neq s$)}}  
   
   \
   
 \n  We show that   $[\bar{ v}] = \{ \bar x_r^\alpha \bar x_s^\beta \colon 0< \alpha, \beta < p\}$, and hence this equivalence class has size $(p-1)^2$. To this end we verify:

    \n \emph{Claim.}    Let $w= \pp  \beta k E$. Let  $\alpha, \beta \neq 0$. We have \bc  $[w, \bar x_r^\alpha \bar  x_s^{\beta} ]= 1 \lra E \sub \{r,s\}$.  \ec
  For the implication from left to right,   if there is $k \in E\setminus \{r,s\}$, then $kr $ and $ks$, so   $A$ has a triangle.  Hence $E \sub \{r,s\}$.
  The converse implication follows by distributivity  and since $[x_r,x_s]=1$. This shows the claim.  
  
  By the claim, $C(\bar x_r^\alpha \bar  x_s^{\beta})= C(v)$ for each $\alpha, \beta \neq 0$. On the other hand, if $K \subset  \{r,s\}$ then $C(\pp \beta  i K)  \not \sub C(v)$ by niceness, and if  $K \not \sub   \{r,s\}$ then $v \not \in C(\pp \beta  i K)$ by the claim again.

%  \n \emph{Claim 2.}  
   
   \subsection*{Case 3:  {\rm  Neither Case 1 nor 2, and there is an $\ell$ such that $i\ell$ for each $i \in D$.} }  \
 
  \n \emph{Claim.}     $[\bar{ v}] = \{ \bar v^\gamma \bar x_\ell^\beta \colon 0 < \gamma < p \lland 0\le   \beta < p\}$, and hence  $|[\bar{ v}]|= p(p-1)$.

\n First an observation: \emph{suppose that $[v^\gamma  x_\ell^\beta,w]=1$ where  $\gamma, \beta$ are as above and $w= \pp \beta k E$. Then  $E \sub D \cup \si \ell$.} For, since   Cases 1 and  2 don't apply, there are distinct $p,q \in D \setminus \si \ell$. Since $A$ has no squares,  $\ell$ is the only vertex adjacent to both $p$ and $q$. Hence, given  $j \in E \setminus \{ \ell,p,q\}$, we have $\lnot pj \lor \lnot qj$, say the former. Then $j \in D$, for otherwise, in Lemma~\ref{lem:com}, in the expansion of $[v^\gamma x_\ell^\beta, w]$, we get a term $x_{p,j}^{m}$ with  $m \neq 0$ and $\lnot pj$ (assuming that $p< j$, say).

  The inclusion  ``$\supseteq$" of the claim now follows  by Lemma~\ref{lem:com}.
  %  \bc $[v,w]=1 \lra [v^\gamma, w]=1 \lra [v^\gamma, w][v^\gamma, x^\beta]=1 \lra [v^\gamma x^\beta, w]=1$. \ec 
    %
    For the inclusion ``$\sub$"    suppose that $[v,w]=1$ where $w= \pp \beta k E$. Then  $E \sub D \cup \si \ell$ by our observation.       
  By Lemma~\ref{lem:com} and the case hypothesis we now have 
  \bc $[v,w] = \prod_{r,s \in D \setminus \si \ell, \,  r< s, \, \lnot rs} x_{r,s}^{\alpha_r \beta_s- \alpha_s \beta_r}$. \ec
Let  $m = \min (D \setminus \si \ell)$. Since $\alpha_m \neq 0$ we can pick $\gamma$ such that $\beta_m= \gamma \alpha_m$. Since $A$ has no  triangles we have $\lnot rs$ for any $r< s$ such that  $r, s \in D \setminus \si \ell$, and hence $\alpha_r \beta_s = \alpha_s \beta_r$.  By induction on the elements $s$ of $D \setminus \si \ell$ this yields $\beta_s= \gamma \alpha_s $: if we have it for some  $r<s$ in $D \setminus \si \ell$ then $\beta_r \beta_s = \gamma \alpha_r\beta_s = \gamma \alpha_s \beta_r$.  Hence 
  $\beta_s= \gamma \alpha_s $, for if $\beta_r\neq 0$ we can cancel it, and if $\beta_r = 0$ then also $\beta_s=0$ because $\alpha_r \neq 0$. 
  
  This shows that $\bar w=  \bar v^\gamma \bar x_\ell^\beta$ for some $\beta$. In particular, $C(v)= C(w)$ implies that $\bar w$ has the required form.

     \subsection*{Case 4: {\rm   Neither Case 1, 2 or 3}} \
  
  \n \emph{Claim.}  $[\bar{ v}] = \{ \bar v^\gamma   \colon 0 < \gamma < p\}$, so this class  has $p-1$ elements.  Further, there is no $\bar w$ such that $[\bar{ v}] R [\bar{ w}]$.
  
 The inclusion ``$\supseteq$" of the first statement follows from Lemma~\ref{lem:centr}. We now verify the converse inclusion and the second statement. By case hypothesis  there are distinct $\ell_0, \ell_1\in D$ such that $\lnot \ell_0 \ell_1$. Since $A$ has no squares  there is at most one $q \in D$ such that $\ell_0q \lland \ell_1 q$.  If $q$ exists, as we are not in Case 3 we can choose $q' \in D$ such that $\lnot q' q$.  
 
 We define a linear order $\prec$ on  $D$, which is  of type $\omega$ if $D$ is infinite. It begins  with  $\ell_0, \ell_1$, and is followed by $q', q$ if they are defined. After that we proceed in ascending order for the remaining elements of $D$. Then  for each $v \in D \setminus \si {\min D}$ there is $u \prec v$ in $D$ (in fact among the first three elements)  such that $\lnot uv$.
  
  Suppose now that $[v,w]=1$ where $w= \pp \beta k E$. Then $E \sub D$: if $s \in E \setminus D$ there is $r \in D$ such that  $\lnot rs$. This implies $\alpha_r \beta_s= \alpha_s \beta_r$, but $\alpha_s= 0$ while the left hand side is $\neq 0$, contradiction.
  
  By a slight variant of Lemma~\ref{lem:com}, using that $\prec$ eventually agrees with $<$,   we now have $[v,w] = \prod_{r,s \in D ,  \,  r\prec  s, \, \lnot rs} x_{r,s}^{\alpha_r \beta_s- \alpha_s \beta_r}$. Choose $\gamma$ such that $\gamma \alpha_{\ell_0} = \beta_{\ell_0}$. By induction along $(D, \prec)$ we see that $\beta_s = \gamma \alpha_s$ for each $s\in D$: if $\ell_0 \prec s$ choose $r \prec s$ such that $\lnot rs$.  Since $\alpha_r \beta_s= \alpha_s \beta_r$,  as in Case 3  above we may conclude that $\beta_s = \gamma \alpha_s$.
  
  This shows that $\bar v^\gamma = \bar w$. 
   Further, if   $[\bar w] \neq [\bar 1]$   then $\gamma \neq 0$ so that  $[\bar v] = [\bar w]$, as required. 
   
     \begin{lemma} For nice graphs $A, B$  we have  $A \cong B$ iff $\hat G(A) \cong \hat G(B)$. 
 \end{lemma}
 For the forward  implication, suppose that $ A \cong B $ via a permutation $\rho \in \Sinf$. Then  $G(A) \cong G(B)$ via the isomorphism induced by  viewing $\rho$  as a permutation of generators.  Clearly $\+ V = \{ R_n \colon  n \in \NN\}$ and $\{ \rho(R_n)  \colon  n \in \NN\}$ are equivalent  bases of neighbourhoods of $1$, So by the remark at then end of  Subsection~\ref{ss:completion} the completions are homeomorphic. 
 
 The backward implication is immediate from Lemma~\ref{lem:recover}.

\end{proof}

It is important in the above argument to allow infinitely generated groups: For profinite groups that are (topologically) finitely generated, the isomorphism relation is smooth, i.e., Borel reducible to the identity relation on $\RR$. To see this, we rely  on the result that two finitely generated profinite groups are isomorphic iff they have the same finite quotients (up to isomorphism); see \cite[Prop.\ 16.10.7]{Fried.Jarden:06}. Note that to a finitely generated profinite group, one  can in a  Borel  way assign the set of isomorphism types of  its finite quotients.  For an alternative proof not relying on \cite[Prop.\ 16.10.7]{Fried.Jarden:06} see the preprint  \cite[Thm.\ 3.1]{Nies:16}. Also see \cite{Nies:16} for the observation that  identity on $\RR$ is Borel reducible  to isomorphism of finitely generated profinite groups.

\def\cprime{$'$} \def\cprime{$'$}

%
%\bibliographystyle{plain}
% \bibliography{../../Logicsharing/bibs/Nies,../../Logicsharing/bibs/randomness,../../Logicsharing/bibs/settheory,../../Logicsharing/bibs/various,../../Logicsharing/bibs/recursiontheory,../../Logicsharing/bibs/analysis,../../Logicsharing/bibs/Kucera,../../Logicsharing/bibs/modeltheory,../../Logicsharing/bibs/reverse_maths,../../Logicsharing/bibs/groups,../../Logicsharing/bibs/ergodic_theory}
%

\end{document}